\documentclass{amsart}
\numberwithin{equation}{section}

  \usepackage{verbatim}
\usepackage{mathrsfs}
\usepackage{latexsym}
\usepackage{epsfig}
\usepackage{amsmath}
\usepackage{bbm}
\usepackage{graphics}
\usepackage{amssymb}
\usepackage{amsthm}
\usepackage[alphabetic]{amsrefs}
\usepackage{amsopn}
\usepackage{amscd}
\usepackage[all,knot]{xy}
\xyoption{all}
\usepackage{rotating}
\usepackage{hyperref}

\theoremstyle{plain}
\newtheorem{thm}{Theorem}[section]
\newtheorem{lem}[thm]{Lemma}
\newtheorem{prop}[thm]{Proposition}
\newtheorem{cor}[thm]{Corollary}

\newtheorem*{thm*}{Theorem}
\newtheorem*{lem*}{Lemma}
\newtheorem*{prop*}{Proposition}
\newtheorem*{cor*}{Corollary}

\theoremstyle{definition}
\newtheorem{defn}[thm]{Definition}
\newtheorem*{defn*}{Definition}

\newtheorem{ex}[thm]{Example}
{}
\newtheorem{rem}[thm]{Remark}
\newtheorem*{rem*}{Remark}

{}
{}
{}
{}

\theoremstyle{remark}
{}
{}
{}

\def\cie{\subseteq}

\def\iso{\cong}

\def\to{\longrightarrow}

\def\int{\mathbb{Z}}

\def\e{\epsilon}

\def\D{\mathsf{D}}
\def\K{\mathsf{K}}

\def\sfT{\mathsf{T}}
\def\sfS{\mathsf{S}}

\def\sfP{\mathsf{P}}
\def\sfM{\mathsf{M}}

\def\sfL{\mathsf{L}}

\def\Ord{\mathsf{Ord}}
\def\GGamma{\mathit{\Gamma}}

\DeclareMathOperator{\Spec}{Spec}

\DeclareMathOperator{\Spc}{Spc}

\DeclareMathOperator{\supp}{supp}

\DeclareMathOperator{\Hom}{Hom}

\DeclareMathOperator{\hocolim}{hocolim}

\DeclareMathOperator{\kdim}{\mathrm{Kdim}}
\DeclareMathOperator{\cbrk}{\mathrm{CBrk}}
\DeclareMathOperator{\Loc}{\mathrm{loc}}
\DeclareMathOperator{\Thick}{\mathrm{thick}}

\title[The local-to-global principle via dimension functions]{The local-to-global principle for triangulated categories via dimension functions}
\author{Greg Stevenson}
\address{Greg Stevenson, Universit\"at Bielefeld, Fakult\"at f\"ur Mathematik, BIREP Gruppe, Postfach 10\,01\,31, 33501 Bielefeld, Germany.}
\email{gstevens@math.uni-bielefeld.de}

\thanks{This research was partially supported by a fellowship from the Alexander von Humboldt Foundation.}

\begin{document}



\begin{abstract}
\noindent 
We formulate a general abstract criterion for verifying the local-to-global principle for a rigidly-compactly generated tensor triangulated category. Our approach is based upon an inductive construction using dimension functions. Using our criterion we give a new proof of the theorem that the local-to-global principle holds for such categories when they have a model and the spectrum of the compacts is noetherian. As further applications we give a new set of conditions on the spectrum of the compacts that guarantee the local-to-global principle holds and use this to classify localising subcategories in the derived category of a semi-artinian absolutely flat ring.
\end{abstract}

\maketitle

\tableofcontents

\section{Introduction}

The local-to-global principle is a powerful and striking property of a support theory on a triangulated category. Roughly speaking, it asserts that given some notion of support for objects of a triangulated category $\sfT$, taking values in a topological space $X$, one can reconstruct any object from the ``pieces of that object supported at points of $X$''.  It was originally christened and formulated in support theoretic language by Benson, Iyengar, and Krause in \cite{BIKStrat2}, although it is worth noting the idea already appears in Neeman's work on localising subcategories in the unbounded derived category of a commutative noetherian ring \cite{NeeChro} (in particular see his Lemma~2.10). The validity of the local-to-global principle has many deep consequences. For instance it implies that the notion of support in question is fine enough to detect whether or not an object is zero, reduces the computation of the lattice of (sufficiently nice) localising subcategories to understanding certain distinguished localising subcategories, and provides a good theory of functorial filtrations and Postnikov towers, with respect to a dimension function on the associated support variety, for the triangulated category. It is thus a fundamental question to determine in which contexts one has access to the local-to-global principle; this is precisely the question with which we concern ourselves in this article.

Let us now be a bit more concrete: let $\sfT$ be a rigidly-compactly generated tensor triangulated category and denote by $\Spc\sfT^c$ the spectrum, in the sense of Balmer \cite{BaSpec}, of its compact objects. If $\Spc \sfT^c$ satisfies certain ``reasonableness conditions'' then canonically associated to every point $x\in \Spc\sfT^c$ there is a tensor idempotent object $\mathit{\Gamma}_x\mathbf{1}$ such that tensoring an object with $\mathit{\Gamma}_x\mathbf{1}$ gives a sort of ``$x$-torsion and $x$-local'' approximation of that object. By definition the local-to-global principle holds if for every $A\in \sfT$ there is an equality of localising tensor ideals
\begin{displaymath}
\Loc^\otimes(A) = \Loc^\otimes(\mathit{\Gamma}_x\mathbf{1}\otimes A \; \vert \; x\in \Spc\sfT^c)
\end{displaymath}
i.e.\ the smallest localising tensor ideal containing $A$ coincides with the smallest localising tensor ideal containing the $\mathit{\Gamma}_x\mathbf{1}\otimes A$. 

In \cite{StevensonActions} Proposition~6.8 it was shown that the local-to-global principle holds provided $\sfT$ arises as the homotopy category of a monoidal model category and $\Spc \sfT^c$ is a noetherian topological space. On the other hand in \cite{StevensonAbsFlat} it was shown, by somewhat ad hoc methods, that there are examples where the spectrum of the compacts is \emph{not} noetherian but the local-to-global principle holds. The aim of this note is to rectify this situation by providing a general abstract criterion, namely Theorem~\ref{thm_abstract}, framed in the language of dimension functions on spectral spaces, to check the validity of the local-to-global principle. The proof is of a somewhat different flavour to the author's original proof and is based upon the inductive construction of objects from the filtrations provided by the dimension functions. It is closer in spirit to the proof of \cite{BIKStrat2}*{Theorem~3.4} and the constructions in \cite{StevensonFilt}. In fact it gives a slightly sharper statement, showing that one can build $A$ from the $\mathit{\Gamma}_x\mathbf{1}\otimes A$ without using tensor products.

Using this abstract criterion we give a new proof in Theorem~\ref{thm_noeth} of the local-to-global principle when $\Spc \sfT^c$ is noetherian and in Theorem~\ref{thm_absflat} we prove the local-to-global principle under a different set of assumptions on $\Spc \sfT^c$ generalising the examples of \cite{StevensonAbsFlat}. As an amusing aside our methods allow one to drop the assumption that $\sfT$ comes from a stable model category in many cases.

As an application of Theorem~\ref{thm_absflat} we complete the classification problem that was studied in \cite{StevensonAbsFlat}. In that article a classification of the localising subcategories of $\mathsf{D}(R)$ was given in the case $R$ is the absolutely flat approximation of a commutative ring with noetherian spectrum. Such rings are semi-artinian\textemdash{}every non-zero module has non-zero socle. On the other hand, it was shown that if $R$ is an absolutely flat ring which is not semi-artinian then the local-to-global principle fails. In Theorem~\ref{thm_app} we complete this to an honest dichotomy, proving that $\mathsf{D}(R)$ satisfies the local-to-global principle for any semi-artinian absolutely flat ring. Moreover, for such rings $R$ we use this to show that the localising subcategories of $\mathsf{D}(R)$ are in bijection with subsets of $\Spec R$.

\section{Preliminaries}
This section contains some brief reminders on various results, constructions, and topics that will be of use in the sequel together with some pertinent references. We begin with some fairly detailed recollections on the theory of spectral spaces. We then discuss, mostly to fix notation, some notions from tensor triangular geometry.

\subsection{Spectral spaces}\label{sec_spectralprelims}
Spectral spaces form the natural topological framework for support theory in tensor triangulated categories; by Stone duality they correspond to coherent frames and these are precisely the lattices which appear when studying tensor ideals in essentially small tensor triangulated categories. They are thus also the natural context for investigating the local-to-global principle. We give here the relevant definitions as well as fixing notation and recording some results we will need.

\begin{defn}\label{defn_spectral}
Let $X$ be a topological space. We say $X$ is a \emph{spectral space} if it is quasi-compact, $T_0$, has a basis of quasi-compact open subsets closed under finite intersections, and every non-empty irreducible closed subset has a (necessarily unique) generic point. A morphism $f\colon Y\to X$ of spectral spaces, namely a continuous map of topological spaces, is \emph{spectral} if it is quasi-compact, i.e.\ the preimage of every quasi-compact open subset of $X$ is quasi-compact in $Y$. We say that $Y$ is a spectral subspace of $X$ if $Y$ is a subspace of $X$ such that the inclusion is spectral.
\end{defn}

Given a spectral space $X$ and $x\in X$ we denote the closure of $x$ by $\mathcal{V}(x)$ and say $x$ \emph{specialises} to $y$ if $y\in \mathcal{V}(x)$. We also fix notation for the subset
\begin{displaymath}
\mathcal{Z}(x) = \{y\in X \; \vert \; x\notin \mathcal{V}(y)\},
\end{displaymath}
which is just the set of points not specialising to $x$.

\begin{defn}\label{defn_thomason}
A subset $\mathcal{V}\subseteq X$ is \emph{Thomason} if $\mathcal{V}$ is a (possibly infinite) union of closed subsets with quasi-compact open complements.
\end{defn}

\begin{ex}\label{ex_thomason}
Both $X$ and $\varnothing$ are always Thomason subsets. Further examples are given by the subsets $\mathcal{Z}(x)$ for $x\in X$\textemdash{}these are Thomason as the complement of $\mathcal{Z}(x)$ is just the intersection of the quasi-compact opens containing $x$. If $X$ is noetherian then the Thomason subsets are just the specialisation closed ones.
\end{ex}

We next discuss the constructible topology on a spectral space.

\begin{defn}\label{defn_cons}
The \emph{constructible topology} on $X$ is generated by the collection of quasi-compact open subsets of $X$ and their complements as a subbasis of open subsets. We denote $X$ together with this topology by $X^\mathrm{con}$. Given a subset $\mathcal{V}\subseteq X$ we say it is \emph{proconstructible} if it is closed in the constructible topology. We say that $X$ carries the constructible topology if $X = X^\mathrm{con}$.
\end{defn}

We now state some important facts concerning the constructible topology on $X$.

\begin{prop}\label{prop_cons}
Let $X$ be a spectral space. The space $X^\mathrm{con}$ is a quasi-compact Hausdorff spectral space such that
\begin{displaymath}
(X^\mathrm{con})^\mathrm{con} = X^\mathrm{con}.
\end{displaymath}
Furthermore, given a subset $Y$ of $X$ the following two conditions are equivalent:
\begin{itemize}
\item[$(1)$] $Y$ is proconstructible, i.e.\ closed in $X^\mathrm{con}$;
\item[$(2)$] $Y$ with the subspace topology is spectral and $Y\to X$ is a spectral map, i.e.\ $Y$ is a spectral subspace of $X$.
\end{itemize}
\end{prop}
\begin{proof}
The proof of the first part of the proposition can be found in \cite{stacks-project}*{Tag 08YF}.

We give a (fairly detailed) sketch of the proof of the second part of the proposition, starting with (1) implies (2). Let $Y$ be a proconstructible subset of $X$ and equip it with the subspace topology. First observe that the quasi-compact open subsets of $Y$ are precisely the intersections of quasi-compact open subsets of $X$ with $Y$. Indeed, if $U$ is a quasi-compact open in $X$ then it is proconstructible, and so the open subset $U\cap Y$ of $Y$ is also proconstructible, as $Y$ is. Hence $U\cap Y$ is quasi-compact by virtue of being closed in a quasi-compact Hausdorff space whose topology is finer than the topology on $X$. On the other hand, if $W$ is a quasi-compact open in $Y$, then there exists an open $U$ in $X$ such that $W = U\cap Y$. Since $X$ is spectral we can write $U$ as a union $U_i$, over some index set $I$, of quasi-compact open subsets. As $W$ is quasi-compact with $W\subseteq \cup_i U_i$ there exists some finite set of indices $i_1,\ldots, i_n$ with $W \subseteq \cup_{j=1}^n U_{i_j} = U'$. It follows that $W = U'\cap Y$ with $U'$ quasi-compact as desired. 

We note this already shows that the inclusion $Y\to X$ is quasi-compact. Moreover, the only requirement for $Y$ to be spectral which does not follow immediately is that irreducible closed subsets have generic points. To this end, suppose $V\subseteq Y$ is an irreducible closed subset and denote by $W$ the closure of $V$ in $X$. As $V$ is closed in $Y$ we have $W\cap Y = V$.

Using this, one readily checks that $V$ is compact in $X^\mathrm{con}$ and that $W$ is irreducible in $X$. In particular, $W$ has a unique generic point $x$ and we proceed by showing $x\in V$. If $x\notin V$ then for each $v\in V$ we can find a quasi-compact open subset $U_v$ of $X$ such that $v\notin U_v$ and $x\in U_v$; this is a consequence of $X$ being $T_0$ and having a basis of quasi-compact open subsets, the other possibility that $x\notin U_v$ and $v\in U_v$ being excluded by the fact that this would imply $v\notin \mathcal{V}(x) = W$. Thus
\begin{displaymath}
V \subseteq \bigcup_{v\in V} X\setminus U_v
\end{displaymath}
and it follows from compactness of $V$ in $X^\mathrm{con}$ that there are a finite set $v_1,\ldots, v_n$ of points in $V$ with
\begin{displaymath}
V \subseteq \bigcup_{i=1}^n X\setminus U_{v_i}.
\end{displaymath}
However, irreducibility of $W$ together with the fact that the right hand side of the above expression is closed implies $W$ would then be contained in some $X\setminus U_{v_i}$ which is absurd, as none of these closed subsets contains $x$. Thus $x\in V$ and we see that $V$ is the closure of $\{x\}$ in $Y$ as desired.

In order to show (2) implies (1) one checks that $Y^\mathrm{con} \stackrel{i}{\to} X^\mathrm{con}$ is again continuous (we omit the details, but they can be found in \cite{stacks-project}*{Tag 08YF}). Given this one just observes that $Y^\mathrm{con}$ is compact, hence its image under $i$ is compact and thus closed in $X^\mathrm{con}$, i.e.\ proconstructible, since the latter space is Hausdorff.
\end{proof}

Lastly we introduce visible points, the relevance of which will become clear when we discuss supports below.

\begin{defn}
We say a point $x\in X$ is \emph{visible} if there exist Thomason subsets $\mathcal{V}$ and $\mathcal{W}$ such that
\begin{displaymath}
\{x\} = \mathcal{V} \setminus (\mathcal{V} \cap \mathcal{W}).
\end{displaymath}
\end{defn}

\begin{lem}\label{lem_consvis}
Let $X$ be a spectral space with $X = X^\mathrm{con}$. Then every point of $X$ is visible.
\end{lem}
\begin{proof}
As noted in Example~\ref{ex_thomason} for every $x\in X$ the subset $\mathcal{Z}(x)$ is Thomason. As $X$ is Hausdorff every point is closed, i.e.\ there are no non-trivial specialisation relations, so $\mathcal{Z}(x) = X\setminus \{x\}$. We can thus take $X$ itself and $\mathcal{Z}(x)$ as a pair of Thomason subsets witnessing the visibility of $x$. 
\end{proof}

The main abstract fact we will need concerning visible points is that visibility passes to subspaces.

\begin{lem}\label{lem_vis}
Let $X$ be a spectral space and let $i\colon Y\to X$ be a spectral subspace. If $y\in Y$ is visible as a point of $X$, i.e.\ $i(y)$ is visible, then $y$ is visible in $Y$. In particular, if every point of $X$ is visible then every point of $Y$ is visible.
\end{lem}
\begin{proof}
Let $X$, $Y$, and $y$ be as in the statement, i.e.\ $y$ is a visible point of $X$ lying in the proconstructible subset $Y$, and choose Thomason subsets $\mathcal{V}$ and $\mathcal{W}$ of $X$ such that $\{x\} = \mathcal{V} \setminus (\mathcal{V}\cap \mathcal{W})$. Since $i\colon Y\to X$ is spectral the preimage of a closed subset of $X$ with quasi-compact complement is such a subset of $Y$. We immediately deduce from this observation that $i^{-1}\mathcal{V}$ and $i^{-1}\mathcal{W}$ are Thomason subsets of $Y$ which then witness the visibility of $y$ as a point of $Y$. 
\end{proof}

\subsection{Tensor triangular geometry}
Let us now give a whirlwind tour of some aspects of support theory for rigidly-compactly generated tensor triangulated categories. We define the main players and the functors which we will use throughout but are sparing with the motivation and details. A thorough treatment of these ideas can be found in \cite{BaRickard} and \cite{StevensonActions}.

\begin{defn}
A \emph{tensor triangulated category} is a triple $(\K, \otimes, \mathbf{1})$, where $\K$ is a triangulated category and
\begin{displaymath}
-\otimes-\colon \K\times \K \to \K
\end{displaymath}
is a symmetric monoidal structure on $\K$ with unit $\mathbf{1}$ and with the property that, for all $k\in \K$, the endofunctors $k\otimes -$ and $-\otimes k$ are exact.
\end{defn}

\begin{rem}
Throughout we shall not make explicit the associativity, symmetry, and unit constraints for the symmetric monoidal structure on a tensor triangulated category. By standard coherence results for monoidal structures this will not get us into any trouble.
\end{rem}

\begin{defn}\label{defn_rigid}
Let $\K$ be an essentially small tensor triangulated category. Assume that $\K$ is closed symmetric monoidal, i.e.\ for each $k\in \K$ the functor $k\otimes-$ has a right adjoint which we denote $\hom(k,-)$. These functors can be assembled into a bifunctor $\hom(-,-)$ which we call the \emph{internal hom} of $\K$. By definition one has, for all $k,l,m \in \K$, the tensor-hom adjunction
\begin{displaymath}
\K(k\otimes l, m) \cong \K(l, \hom(k,m)),
\end{displaymath}
with corresponding units and counits
\begin{displaymath}
\eta_{k,l} \colon l \to \hom(k,k\otimes l) \quad \text{and} \quad \epsilon_{k,l} \colon \hom(k,l)\otimes k \to l.
\end{displaymath}
The \emph{dual} of $k\in \K$ is the object
\begin{displaymath}
k^\vee = \hom(k,\mathbf{1}).
\end{displaymath}
Given $k,l \in \K$ there is a natural \emph{evaluation map}
\begin{displaymath}
k^\vee \otimes l \to \hom(k,l),
\end{displaymath}
which is defined by following the identity map on $l$ through the composite
\begin{displaymath}
\xymatrix{
\K(l,l) \ar[r]^-\sim & \K(l\otimes \mathbf{1}, l) \ar[rr]^-{\K(l\otimes \e_{k,\mathbf{1}}, l)} && \K(k\otimes k^\vee \otimes l, l) \ar[r]^-\sim & \K(k^\vee \otimes l, \hom(k,l)).
}
\end{displaymath}
We say that $\K$ is \emph{rigid} if for all $k,l \in \K$ this natural evaluation map is an isomorphism
\begin{displaymath}
k^\vee \otimes l \stackrel{\sim}{\to} \hom(k,l).
\end{displaymath}
\end{defn}

Our main interest will be in certain compactly generated tensor triangulated categories, namely the rigidly-compactly generated ones.

\begin{defn}
A \emph{rigidly-compactly generated tensor triangulated category} $\sfT$ is a compactly generated tensor triangulated category such that the full subcategory $\sfT^c$ of compact objects is rigid. Explicitly, not only is $\sfT^c$ a tensor subcategory of $\sfT$, but it is closed under the internal hom (which exists by Brown representability) and is rigid in the sense of Definition~\ref{defn_rigid}.
\end{defn}

Fix a rigidly-compactly generated tensor triangulated category $\sfT$. We say that a full replete subcategory $\sfL$ of $\sfT$ is \emph{localising} if it is closed under suspensions, cones, and arbitrary coproducts and it is a \emph{localising tensor ideal} if it is in addition closed under tensoring with arbitrary objects of $\sfT$. A full replete subcategory $\sfM$ of $\sfT^c$ is \emph{thick} if it is closed under suspensions, cones, and direct summands and one defines a thick tensor ideal in the obvious way. Given classes of objects $S\subseteq \sfT$ and $S'\subseteq \sfT^c$ we denote by
\begin{displaymath}
\Loc(S), \; \Loc^\otimes(S), \; \Thick(S'), \; \text{and}\; \Thick^\otimes(S')
\end{displaymath}
the smallest localising subcategory, localising tensor ideal, thick subcategory, and thick tensor ideal containing $S$ or $S'$ respectively.

As in \cite{BaSpec} we can associate to $\sfT^c$ its spectrum $\Spc \sfT^c$ which is the set of \emph{prime tensor ideals} in $\sfT^c$ i.e.\ those proper thick tensor ideals $\sfP$ such that if $t\otimes t'\in \sfP$ then one of $t$ or $t'$ lies in $\sfP$. This is endowed with a topology by declaring a basis of closed subsets to be given by the \emph{supports} of objects $t\in \sfT^c$:
\begin{displaymath}
\supp t = \{\sfP\in \Spc\sfT^c \; \vert \; t\notin \sfP\}.
\end{displaymath}

A key fact, which we exploit throughout the whole paper, is the following theorem. Although it is not stated explicitly in Balmer's paper \cite{BaSpec} all of the ingredients are there. It was later generalised in work of Buan, Krause, and Solberg \cite{BKS}.

\begin{thm}
The space $\Spc \sfT^c$ is spectral in the sense of Definition~\ref{defn_spectral}.
\end{thm}

Next let us recall the generalised Rickard idempotents which allow us to extend the support to arbitrary objects of $\sfT$ (in favourable situations) and introduce the local-to-global principle.

Let $\mathcal{V}$ be a Thomason subset of $\Spc\sfT^c$. There is, by \cite{BaSpec}*{Theorem~4.10}, a unique thick tensor ideal in $\sfT^c$ associated to $\mathcal{V}$, namely
\begin{displaymath}
\sfT^c_\mathcal{V} = \{t\in \sfT^c \; \vert \; \supp t \cie \mathcal{V}\}.
\end{displaymath}
We set
\begin{displaymath}
\mathit{\Gamma}_\mathcal{V}\sfT = \Loc(\sfT^c_\mathcal{V}).
\end{displaymath}
There is an associated smashing localisation sequence
\begin{displaymath}
\xymatrix{
\mathit{\Gamma}_\mathcal{V}\sfT \ar[r]<0.5ex> \ar@{<-}[r]<-0.5ex> & \sfT \ar[r]<0.5ex> \ar@{<-}[r]<-0.5ex> & L_\mathcal{V}\sfT 
}
\end{displaymath}
and we denote the corresponding acyclisation and localisation functors by $\mathit{\Gamma_\mathcal{V}}$ and $L_\mathcal{V}$ respectively. By \cite{BaRickard}*{Theorem~4.1} $\mathit{\Gamma}_\mathcal{V}\sfT$ is not only a smashing subcategory but a smashing tensor ideal. Thus there are associated tensor idempotents $\mathit{\Gamma}_\mathcal{V}\mathbf{1}$ and $L_\mathcal{V}\mathbf{1}$ which give rise to the acyclisation and localisation functors by tensoring. It is these idempotents that are used to define the support; the intuition is that, for $X\in \sfT$, the object $\mathit{\Gamma}_\mathcal{V}\mathbf{1} \otimes X$ is the ``piece of $X$ supported on $\mathcal{V}$'' and $L_\mathcal{V}\mathbf{1}\otimes X$ is the ``piece of $X$ supported on the complement of $\mathcal{V}$''.

\begin{defn}\label{defn_gen_ptfunctors2}
Let $x$ be a visible point of $\Spc \sfT^c$ and choose Thomason subsets $\mathcal{V}$ and $\mathcal{W}$ such that
\begin{displaymath}
\{x\} = \mathcal{V}\setminus (\mathcal{V}\cap \mathcal{W}).
\end{displaymath}
We define a tensor-idempotent
\begin{displaymath}
\mathit{\Gamma}_x\mathbf{1}= (\mathit{\Gamma}_{\mathcal{V}}\mathbf{1} \otimes L_{\mathcal{W}}\mathbf{1}).
\end{displaymath}
Following the intuition above, for an object $X\in \sfT$, the object $\mathit{\Gamma}_x\mathbf{1}\otimes X$ is supposed to be the ``piece of $X$ which lives only over the point $x\in \Spc\sfT^c$''.
\end{defn}

\begin{rem}\label{rem_bik6.2_uber}
Given any other Thomason subsets $\mathcal{V}'$ and $\mathcal{W}'$ of $\Spc \sfT^c$ with the property that $\mathcal{V}' \setminus (\mathcal{V}'\cap \mathcal{W}') = \{x\}$ we can define a similar object by forming the tensor product $\mathit{\Gamma}_{\mathcal{V}'}\mathbf{1} \otimes L_{\mathcal{W}'}\mathbf{1}$. By \cite{BaRickard}*{Corollary~7.5} any such object is uniquely isomorphic to $\mathit{\Gamma}_x\mathbf{1}$.
\end{rem}

The objects $\mathit{\Gamma}_x\mathbf{1}$ are used to define a notion of support on all of $\sfT$. We will assume from this point onward that every point of $\Spc\sfT^c$ is visible\textemdash{}this guarantees that there are enough of the objects $\mathit{\Gamma}_x\mathbf{1}$. We will, given $A\in \sfT$, often write $\mathit{\Gamma}_xA$ as shorthand for $\mathit{\Gamma}_x\mathbf{1}\otimes A$. 

\begin{defn}
For $A\in \sfT$ we define the \emph{support} of $A$ to be
\begin{displaymath}
\supp A = \{x\in \Spc\sfT^c \; \vert \; \mathit{\Gamma}_xA \neq 0\}.
\end{displaymath}
\end{defn}

Our main interest in this article is the following definition.

\begin{defn}\label{defn_ltg}
We say $\sfT$ satisfies the \emph{local-to-global principle} if for each $A$ in $\sfT$
\begin{displaymath}
\Loc^\otimes(A) = \Loc^\otimes(\mathit{\Gamma}_x A \; \vert \; x\in \Spc \sfT^c).
\end{displaymath}
\end{defn}

The local-to-global principle was originally introduced by Benson, Iyengar, and Krause in \cite{BIKStrat2} and considered in the general form stated above in \cite{StevensonActions}. In the latter paper it is shown that it holds in significant generality provided $\Spc \sfT^c$ is noetherian. The aim of this paper is to weaken this hypothesis by giving a different argument, which actually proves a slightly stronger statement. What the arguments have in common is the need for a reasonable theory of homotopy colimits. We close the preliminaries with a quick discussion of what is necessary to ensure we have access to such homotopy colimits.

\begin{defn}\label{defn_hocolim}
We will say the tensor triangulated category $\sfT$ \emph{has a (monoidal) model} if it occurs as the homotopy category of a stable monoidal model category.
\end{defn}

If $\sfT$ has a model in the above sense then we can use the rich theory of homotopy colimits that have been developed in the framework of stable model categories.

\begin{rem}
Of course instead of requiring that $\sfT$ arose from a stable monoidal model category we could, for instance, ask that $\sfT$ was the underlying category of a stable monoidal derivator. In fact we will only use homotopy colimits of shapes given by ordinals so one could use a weaker notion of a stable monoidal ``derivator'' only having homotopy left and right Kan extensions for certain diagrams; to be slightly more precise one could just ask for homotopy left and right Kan extensions along the smallest full $2$-subcategory of the category of small categories satisfying certain natural closure conditions and containing the ordinals (one can see the discussion before \cite{GrothPSD} Definition 4.21 for further details).
\end{rem}

\section{Dimension functions and the general criterion}
In this section we introduce spectral dimension functions and use them to formulate a general inductive approach to proving the local-to-global principle. We begin by introducing the dimension functions we will consider. Let us fix some spectral space $X$ and denote by $\Ord$ the class of ordinals.

\begin{defn}
We call a function $\dim \colon X \to \Ord$ a \emph{dimension function} if
\begin{itemize}
\item[(i)] $\dim$ does not take limit ordinal values i.e.\ $\dim x$ is not a limit ordinal for any $x\in X$;
\item[(ii)] for $x\in X$ and $y\in \mathcal{V}(x)$ we have $\dim y \leq \dim x$;
\item[(iii)] for $x\in X$ and $y\in \mathcal{V}(x)$ such that $\dim y = \dim x$ we have $x=y$.
\end{itemize}

Given a dimension function we set, for $\alpha \in \Ord$,
\begin{displaymath}
X_{\leq \alpha} = \{x\in X\; \vert\; \dim x \leq \alpha\} \quad \text{and} \quad X_{>\alpha} = X\setminus X_{\leq \alpha} = \{x\in X\; \vert \; \dim x > \alpha\}.
\end{displaymath}
We define the \emph{dimension} of $X$ relative to the dimension function to be the least ordinal $\alpha$ such that $X = X_{\leq \alpha}$ and denote it by $\dim X$. Such an ordinal is guaranteed to exist since $X$ has a set of points and the class of ordinals is well-ordered.

Finally we say $\dim \colon X\to \Ord$ is \emph{spectral} if $X_{\leq \alpha}$ is a Thomason subset of $X$ for all $\alpha \in \Ord$.
\end{defn}

Some of the reasoning behind the use of the adjective spectral is demonstrated by the following lemma.

\begin{lem}
Let $X$ be a spectral space and $\dim$ a spectral dimension function on $X$. Then for each ordinal $\alpha$ the subset $X_{>\alpha}$ equipped with the subspace topology is a spectral subspace of $X$.
\end{lem}
\begin{proof}
Let $\alpha$ be an ordinal. By definition $X_{\leq \alpha}$ is Thomason i.e.\ it can be written as a union of closed subsets with quasi-compact open complements. Thus $X_{>\alpha}$ is an intersection of quasi-compact opens. In particular, $X_{>\alpha}$ is proconstructible and hence a spectral subspace of $X$ by Proposition~\ref{prop_cons}.
\end{proof}

In order to give our criterion for the local-to-global principle we will have to consider families of spectral dimension functions on classes of spectral spaces. We next introduce the relevant definitions, starting with the families of spectral spaces we require.

\begin{defn}
Let $\mathcal{X}$ be a class of spectral spaces. We say $\mathcal{X}$ is \emph{subspace closed} if it contains all spectral subspaces of its members. More explicitly, we require that if $X\in \mathcal{X}$ and $Y\to X$ is the inclusion of a spectral subspace of $X$ then $Y\in \mathcal{X}$.
\end{defn}

\begin{rem}
Many interesting properties of spectral spaces give rise to subspace closed classes. For instance, the class of noetherian spectral spaces (Lemma~\ref{lem_noethclosed}), the class of spectral spaces all of whose points are visible (Lemma~\ref{lem_vis}), and the class of spectral spaces carrying the constructible topology and having Cantor-Bendixson rank are all subspace closed (Lemma~\ref{lem_CBclosed}) and see Definition~\ref{defn_cbrank} for details on Cantor-Bendixson rank).
\end{rem}

\begin{defn}
Let $\mathcal{X}$ be a subspace closed class of spectral spaces and 
\begin{displaymath}
\mathcal{D} = \{\dim_X\colon X\to \Ord \; \vert \; X\in \mathcal{X}\}
\end{displaymath}
a class of spectral dimension functions on the spaces in $\mathcal{X}$. We say $\mathcal{D}$ is \emph{compatible} with $\mathcal{X}$ if for all $X\in \mathcal{X}$, $x\in X$, and $\alpha \in \Ord$
\begin{displaymath}
\dim_X(x) = \alpha +1 \quad \text{if and only if} \quad \dim_{X>\alpha}(x) = 0.
\end{displaymath}
\end{defn}

\begin{rem}
In all of our examples compatibility will come for free as our dimension functions will be defined recursively, via some topological property, using the $X_{>\alpha}$.
\end{rem}


We now have enough terminology to state our main abstract result.

\begin{thm}\label{thm_abstract}
Let $\mathcal{X}$ be a subspace closed class of spectral spaces, such that for each $X\in \mathcal{X}$ every point of $X$ is visible, and let $\mathcal{D}$ be a compatible class of spectral dimension functions. Suppose for every rigidly-compactly generated tensor triangulated category $\sfT$ with $Y = \Spc \sfT^c \in \mathcal{X}$ we have for all $A\in \sfT$ that
\begin{displaymath}
\GGamma_{Y_{\leq 0}}A \in \Loc(\GGamma_xA \; \vert \; x\in Y_{\leq 0}).
\end{displaymath}
Then if $\sfS$ is any rigidly-compactly generated tensor triangulated category with a monoidal model and $\Spc \sfS^c\in \mathcal{X}$ we have
\begin{displaymath}
A \in \Loc( \GGamma_x A \; \vert \; x\in \supp A) \;\; \forall A\in \sfS.
\end{displaymath}
\end{thm}

The theorem is, at first glance, perhaps somewhat abstruse. The intuition is that, given a rigidly-compactly generated tensor triangulated category $\sfS$ as in the theorem, the corresponding dimension function gives a filtration of $\Spc \sfS^c$ and hence of $\sfS$ (see \cite{BaFilt} and \cite{StevensonFilt} for more details on such filtrations). This filtration gives a functorial way of building any $A\in \sfS$ from the $\GGamma_xA$ by taking cones and homotopy colimits. The upshot of our compatibility assumption is that inspecting this process one sees that at each successor ordinal $\alpha +1$ a new approximation of $A$ is obtained by attaching copies of $\GGamma_xA$ for points $x$ which are of dimension $0$ in $X_{>\alpha}$. Thus one can deduce the $\GGamma_xA$ build $A$ provided one understands what happens at points of dimension $0$, which is precisely what we assume in the theorem. As we shall see in the examples this hypothesis is not unreasonable as for some naturally occurring dimension functions the $0$ dimensional points are topologically distinguished in some way and so easier to work with.

Before proving the theorem we need a technical lemma which is, in some sense, a generalisation of \cite{StevensonActions}*{Lemma~6.6}. As the proofs are rather similar we only sketch the argument here.

\begin{lem}\label{lem_hocolim}
Let $\sfS$ be a rigidly-compactly generated tensor triangulated category with a monoidal model and suppose $\dim$ is a spectral dimension function on $X = \Spc \sfS^c$. For a limit ordinal $\lambda$ there is, for every object $A$ in $\sfS$, an isomorphism
\begin{displaymath}
\underset{\kappa < \lambda}{\hocolim}\, \GGamma_{X_{\leq \kappa}} A \iso \GGamma_{X_{\leq \lambda}} A.
\end{displaymath}
\end{lem}
\begin{proof}
As $\dim$ does not take limit ordinal values we have
\begin{displaymath}
X_{\leq \lambda} = \bigcup_{\kappa < \lambda} X_{\leq \kappa},
\end{displaymath}
and since $\dim$ is spectral all of the subsets occuring in this expression are Thomason. Thus, by the classification of thick tensor ideals of $\sfS^c$ \cite{BaSpec}*{Theorem~4.10} we have
\begin{displaymath}
\sfS^c_{X_{\leq \lambda}} = \bigcup_{\kappa < \lambda} \sfS^c_{X_{\leq \kappa}}
\end{displaymath}
where the category on the right is a thick tensor ideal by virtue of being a union of a chain of such.
Hence the collection of objects
\begin{displaymath}
\{s\in \sfS^c \; \vert \; \exists\; \kappa < \lambda \;\text{with}\; \supp s \subseteq X_{\leq \kappa} \}
\end{displaymath}
generates $\GGamma_{X_{\leq \lambda}}\sfS$.

For $A\in \sfS$ consider the triangle
\begin{displaymath}
\underset{\kappa < \lambda}{\hocolim}\, \GGamma_{X_{\leq \kappa}} A \to \GGamma_{X_{\leq \lambda}} A \to Z \to \Sigma \,\underset{\kappa < \lambda}{\hocolim}\, \GGamma_{X_{\leq \kappa}} A.
\end{displaymath}
As $\GGamma_{X_{\leq \lambda}}\sfS$ is localising, and hence closed under homotopy colimits, the first two terms of the triangle are contained in $\GGamma_{X_{\leq \lambda}}$ and thus so is $Z$. In order to complete the proof it is sufficient to show that $Z\iso 0$ or equivalently that no object in the generating set for $\GGamma_{X_{\leq \lambda}}\sfS$ we have exhibited above has a non-zero morphism to $Z$. For $\kappa < \lambda$ and $s\in \sfS^c$ with $\supp s \subseteq X_{\leq \kappa}$ we have $s \iso \GGamma_{X_{\leq \kappa}}s$, so by adjunction
\begin{align*}
\Hom(s,Z) &\iso \Hom(\GGamma_{X_{\leq \kappa}}s, Z) \\
&\iso \Hom(s, \GGamma_{X_{\leq \kappa}}Z).
\end{align*}
One then proceeds, as in \cite{StevensonActions}*{Lemma~6.6}, by showing $\GGamma_{X_{\leq \kappa}}Z$ is zero.
\end{proof}

\begin{proof}[Proof of the theorem]
Let $\sfS$ be as in the theorem, set $X=\Spc\sfS^c \in \mathcal{X}$ and denote by $\dim_X$ the associated spectral dimension function. We prove, by transfinite induction, that for every $\alpha \in \Ord$ and $A\in \sfS$ we have
\begin{displaymath}
\GGamma_{X_{\leq \alpha}}A \in \Loc( \GGamma_xA \; \vert \; x\in X_{\leq \alpha}).
\end{displaymath}
By hypothesis the base case $\alpha = 0$ holds.

We begin with the case of successor ordinals: suppose the result holds for ordinals less than or equal to $\alpha$ and consider $X_{\leq \alpha+1}$. As $\dim_X$ is spectral the subset $X_{\leq \alpha}$ is Thomason and so, for $A\in \sfS$, we can consider the localisation triangle
\begin{displaymath}
\GGamma_{X_{\leq \alpha}}\GGamma_{X_{\leq \alpha+1}} A \to \GGamma_{X_{\leq \alpha+1}} A \to L_{X_{\leq \alpha}}\GGamma_{X_{\leq \alpha+1}} A \to \Sigma \GGamma_{X_{\leq \alpha}}\GGamma_{X_{\leq \alpha+1}} A
\end{displaymath}
for the smashing localisation associated to $X_{\leq \alpha}$. We view $L_{X_{\leq \alpha}}\GGamma_{X_{\leq \alpha+1}} A$ as an object of $L_{X_{\leq \alpha}}\sfS$. Note that, since $\mathcal{X}$ is subspace closed and $\Spc L_{X_{\leq \alpha}}\sfS^c \iso X_{> \alpha}$, the base case applies to $L_{X_{\leq \alpha}}\sfS$ as it is again a rigidly-compactly generated tensor triangulated category. In $L_{X_{\leq \alpha}}\sfS$ we have isomorphisms
\begin{displaymath}
L_{X_{\leq \alpha}}\GGamma_{X_{\leq \alpha+1}} A \iso \GGamma_{X_{\leq \alpha+1}}L_{X_{\leq \alpha}} A \iso \GGamma_{{X_{\leq \alpha+1}}\cap X_{> \alpha}}L_{X_{\leq \alpha}} A.
\end{displaymath}
By compatibility of the dimension functions on $\mathcal{X}$ we have
\begin{displaymath}
{X_{\leq \alpha+1}}\cap X_{>\alpha} = (\Spc L_{X_{\leq \alpha}}\sfS^c)_{\leq 0}
\end{displaymath}
so the hypothesis of the theorem applies to yield
\begin{displaymath}
\GGamma_{X_{\leq \alpha+1}}L_{X_{\leq \alpha}} A \in \Loc( \GGamma_x \GGamma_{X_{\leq \alpha+1}} L_{X_{\leq \alpha}}A \; \vert \; x\in {X_{\leq \alpha+1}}\cap X_{> \alpha}) = \Loc( \GGamma_x A \; \vert \; x\in {X_{\leq \alpha+1}}\cap X_{> \alpha}).
\end{displaymath}
On the other hand, by the inductive hypothesis, we have
\begin{displaymath}
\GGamma_{X_{\leq \alpha}}\GGamma_{X_{\leq \alpha+1}}A \iso \GGamma_{X_{\leq \alpha}}A \in \Loc( \GGamma_xA\; \vert \; x\in {X_{\leq \alpha}}) \subseteq \Loc(\GGamma_xA\; \vert \; x\in X_{\leq \alpha +1}).
\end{displaymath}
This shows that both $\GGamma_{X_{\leq \alpha}}\GGamma_{X_{\leq \alpha+1}}A$ and $L_{X_{\leq \alpha}}\GGamma_{X_{\leq \alpha+1}}A$  lie in $\Loc( \GGamma_xA\; \vert \;x\in {X_{\leq \alpha+1}})$ and so, by the localisation triangle, we deduce that $\GGamma_{X_{\leq \alpha+1}}A$ also lies in this localising subcategory as desired.

It remains to prove the induction step for limit ordinals. Suppose then $\lambda$ is a limit ordinal and the inductive hypothesis holds for $\kappa < \lambda$. For $A\in \sfS$ we have, by Lemma~\ref{lem_hocolim}, an isomorphism
\begin{displaymath}
\underset{\kappa < \lambda}{\hocolim}\, \GGamma_{X_{\leq \kappa}} A \iso \GGamma_{X_{\leq \lambda}} A.
\end{displaymath}
By inductive hypothesis each $\GGamma_{X_{\leq \kappa}}A$ lies in $\Loc( \GGamma_xA\; \vert \; x\in X_{\leq \lambda})$. Thus $\GGamma_{X_{\leq \lambda}} A$ is also lies in this subcategory as localising subcategories are closed under homotopy colimits.

The statement of the theorem now follows: as $X$ is a set there exists an ordinal $\beta$ such that $X = X_{\leq \beta}$ and one just applies, for $A\in \sfS$, what we have proved to 
\begin{displaymath}
\GGamma_{X_{\leq \beta}}A = \GGamma_XA \iso A.
\end{displaymath}
\end{proof}

\begin{rem}
One sees from the proof that one could, instead of considering all rigidly-compactly generated tensor triangulated categories with spectrum in $\mathcal{X}$, just consider some quotient closed collection of such tensor triangulated categories.

The theorem can also be weakened as follows: we can ask in the hypothesis only that for any rigidly-compactly generated tensor triangulated category $\sfT$ with $Y = \Spc \sfT^c \in \mathcal{X}$ then given any $A\in \sfT$ we have
\begin{displaymath}
\GGamma_{Y_{\leq 0}}A \in \Loc^\otimes( \GGamma_xA \; \vert \; x\in Y_{\leq 0})
\end{displaymath}
i.e.\ only asking that the local-to-global principle holds for dimension zero points rather than the stronger statement that the $\GGamma_xA$ generate $\GGamma_{Y_{\leq 0}}A$ without using the tensor product.
\end{rem}

The theorem immediately implies the local-to-global principle in the form that was recalled in Definition~\ref{defn_ltg}. As alluded to in the above remark it is in fact stronger as we have exhibited $A$ as an object of the localising subcategory generated by the $\GGamma_xA$ not just the tensor ideal they generate.

\begin{cor}\label{cor_ltg}
Suppose $\sfS$ is as in the theorem, whose conditions we assume to be verified. Then the local-to-global principle holds for $\sfS$. Moreover, for every $A\in \sfS$ we have $\supp A = \varnothing$ if and only if $A\iso 0$. Furthermore these properties hold for the action of $\sfS$ on any other compactly generated triangulated category.
\end{cor}
\begin{proof}
Let $A\in \sfS$. Then, by the theorem, we have
\begin{displaymath}
A\in \Loc( \GGamma_xA\; \vert \; x\in \supp A) \subseteq \Loc^\otimes( \GGamma_xA\; \vert \; x\in \supp A).
\end{displaymath}
On the other hand it is immediate that $\GGamma_xA \iso \GGamma_x\mathbf{1}\otimes A \in \Loc^\otimes( A )$. We thus conclude
\begin{displaymath}
\Loc^\otimes( A ) = \Loc^\otimes( \GGamma_xA \; \vert \; x\in \supp A)
\end{displaymath}
i.e.\ the local-to-global principle holds as claimed.

One then proves the remaining statements exactly as in \cite{StevensonActions}*{Theorem~6.9} using the local-to-global principle.
\end{proof}

As an amusing aside we observe that the formulation of our result also allows one to remove the hypothesis of the existence of a model in certain cases.

\begin{cor}\label{cor_abstract}
Let $\mathcal{X}$ and $\mathcal{D}$ be a subspace closed class of spectral spaces and a compatible collection of dimension functions verifying the hypotheses of the theorem. If $\sfS$ is a rigidly-compactly generated tensor triangulated category with $\Spc \sfS^c \in \mathcal{X}$ and
\begin{displaymath}
\dim(\Spc \sfS^c) < \omega + \omega,
\end{displaymath}
then $A\in \Loc( \GGamma_x A \; \vert \; x\in \supp A )$ for objects $A$ of $\sfS$.
\end{cor}
\begin{proof}
Due to the dimension restriction on $\Spc \sfS^c$ one only needs homotopy colimits indexed by $\omega$ in order for the induction in the proof of the theorem to reach $\Spc \sfS^c$. Since a suitable theory of $\omega$-indexed homotopy colimits exists in any triangulated category with small coproducts the proof of the theorem goes through.
\end{proof}

\begin{rem}
Of course Corollary~\ref{cor_ltg} is also valid when the above corollary holds.
\end{rem}

We close the section with a final abstract lemma which gives a condition under which one can check the hypothesis of the theorem.

\begin{lem}\label{lem_theone}
Let $\mathcal{X}$ be a subspace closed class of spectral spaces and $\mathcal{D}$ a compatible class of spectral dimension functions. Suppose for any $X\in \mathcal{X}$ and $x\in X_{\leq 0}$ the subset $\{x\}$ is Thomason (so in particular closed). Then for any rigidly-compactly generated tensor triangulated category $\sfT$ with $X = \Spc \sfT^c \in \mathcal{X}$ we have
\begin{displaymath}
\GGamma_{X_{\leq 0}}A \in \Loc( \GGamma_xA \; \vert \; x\in X_{\leq 0})
\end{displaymath}
for every $A\in \sfT$.
\end{lem}
\begin{proof}
We have assumed each point in $X_{\leq 0}$ is a Thomason subset and so, by the classification of thick tensor ideals of $\sfT^c$ \cite{BaSpec}*{Theorem~4.10}, we see
\begin{displaymath}
\sfT^c_{X_{\leq 0}} = \Thick( \bigcup_{x\in X_{\leq 0}} \sfT^c_{\{x\}}).
\end{displaymath}
Thus, taken together, the objects of the $\sfT^c_{\{x\}}$ for $x\in X_{\leq 0}$ give a set (up to choosing isomorphism classes) of compact generators for $\GGamma_{X_{\leq 0}}\sfT$.

Let $A$ be an object of $\sfT$. For each point $x\in X_{\leq 0}$ we have a natural map 
\begin{displaymath}
\GGamma_x A = \GGamma_{\{x\}}A = \GGamma_{\mathcal{V}(x)}A \to A,
\end{displaymath}
inducing the first map in the following triangle in $\GGamma_{X_{\leq 0}}\sfT$
\begin{displaymath}
\coprod_{x\in X_{\leq 0}} \GGamma_xA \stackrel{f}{\to} \GGamma_{X_{\leq 0}}A \to Z \to \Sigma \coprod_{x\in X_{\leq 0}} \GGamma_xA.
\end{displaymath}
We will show $Z\iso 0$ proving the first map is an isomorphism and hence proving the lemma. Let $x$ be a point in $X_{\leq 0}$ and $t\in \sfT^c_{\{x\}}$. We have $t\iso \GGamma_xt$ so, by adjunction, every map $t \to \GGamma_{X_{\leq 0}}A$ factors uniquely via the natural map $\GGamma_xA \to \GGamma_{X_{\leq 0}}A$ and hence via the map $f$ in the above triangle. Similarly every map $t\to \coprod_{x\in X_{\leq 0}}\GGamma_xA$ factors uniquely via $\GGamma_xA$ and we thus see $f$ induces an isomorphism
\begin{displaymath}
\Hom(t, \coprod_{x\in X_{\leq 0}} \GGamma_xA) \stackrel{\sim}{\to} \Hom(t, \GGamma_{X_{\leq 0}}A).
\end{displaymath}
From this we deduce that $\Hom(t,Z) \iso 0$. We have already observed above that such $t$ give a set of generators for $\GGamma_{X_{\leq 0}}\sfT$ and so $Z\iso 0$ as claimed. Hence $f$ is an isomorphism
\begin{displaymath}
\GGamma_{X_{\leq 0}}A\cong \coprod_{x\in X_{\leq 0}} \GGamma_xA
\end{displaymath}
and the latter object certainly lies in $\Loc( \GGamma_xA \; \vert \; x\in X_{\leq 0})$.
\end{proof}

\begin{rem}
The lemma is phrased so as to be maximally compatible with the theorem rather than for maximal strength. All that is really needed is that $X=\Spc \sfT^c$ is equipped with a dimension function such that every point in $X_{\leq 0}$ is Thomason. Moreover, one concludes that $\mathit{\Gamma}_{X_{\leq 0}}A$ actually decomposes as a sum of the $\mathit{\Gamma}_xA$. 
\end{rem}

\section{Categories with noetherian spectrum}

This section is devoted to demonstrating that one deduces the local-to-global principle for rigidly-compactly generated tensor triangulated categories with noetherian spectrum and a monoidal model, as originally proved in \cite{StevensonActions}*{Theorem~6.9}, from Theorem~\ref{thm_abstract}. In fact, as previously noted, the version we state and prove here is slightly stronger: we show no tensor products are required to build an object from its local pieces and the need for a model is removed in many cases. This latter fact seems to be more of a curiosity than anything else, but could potentially be of future interest.

To begin let us introduce the dimension function we will be concerned with, namely a transfinite version of the Krull dimension (we note our definition is slightly different from the ordinal valued Krull dimension of \cite{GKrause_Kdim}).

\begin{defn}
Let $X$ be a noetherian spectral space. We define a function
\begin{displaymath}
\kdim_X\colon X \to \Ord
\end{displaymath}
in terms of the subsets $X_{\leq \alpha}$ by transfinite recursion. For $x\in X$ we set $\kdim_X(x) = 0$ if and only if $x$ is closed i.e.\ $X_{\leq 0}$ is just the set of closed points of $X$. For a limit ordinal $\lambda$ we set $X_{\leq \lambda} = \cup_{\kappa<\lambda} X_{\leq \kappa}$. Finally, given $\alpha \in \Ord$ such that $X_{\leq \alpha}$ has been defined we say $\kdim_X(x) = \alpha+1$ if and only if $x$ is a closed point in $X\setminus X_{\leq \alpha} = X_{>\alpha}$ with the subspace topology.
\end{defn}

It is straightforward to show this gives a family of dimension functions of the sort we have considered in the last section.

\begin{lem}
Let $X$ be a noetherian spectral space. Then $\kdim_X$ is a dimension function on $X$.
\end{lem}
\begin{proof}
By construction $\kdim_X$ does not take limit ordinal values. For the remaining two conditions suppose we are given $y\in \mathcal{V}(x)$ with $\kdim_X(x) = \alpha+1$. By definition this means $\mathcal{V}(x)\cap X_{>\alpha} = \{x\}$. So either $y=x$ or $y\neq x$ and then $y\notin X_{>\alpha}$ and thus $\kdim_X(y) < \kdim_X(x)$.
\end{proof}

\begin{lem}
Let $X$ be a noetherian spectral space. For every $\alpha \in \Ord$ the subspace $X_{\leq \alpha}$, defined in terms of $\kdim_X$, is Thomason and so $\kdim_X$ is a spectral dimension function.
\end{lem}
\begin{proof}
As $X$ is noetherian it is enough to show $X_{\leq \alpha}$ is specialisation closed. We proceed by transfinite induction. Clearly $X_{\leq 0}$, which is just the set of closed points of $X$, is specialisation closed which verifies the base case. For the successor ordinal case suppose $X_{\leq \alpha}$ is specialisation closed. We have
\begin{displaymath}
X_{\leq \alpha+1} = X_{\leq \alpha} \cup \{x\in X_{>\alpha} \; \vert \; x \; \text{closed in} \; X_{>\alpha}\},
\end{displaymath}
which is specialisation closed: $x$ is closed in $X_{>\alpha}$ if and only if $\mathcal{V}(x)\cap X_{>\alpha} = \{x\}$ i.e., every specialisation of $x$ lies in the specialisation closed subset $X_{\leq \alpha} \subseteq X_{\leq \alpha+1}$. Finally the limit ordinal case is clear as any union of specialisation closed subsets is specialisation closed.
\end{proof}

In fact it is also easy to see the class of noetherian spectral spaces is subspace closed and the Krull dimension gives a compatible family of dimension functions.

\begin{lem}\label{lem_noethclosed}
If $X$ is a noetherian spectral space and $i\colon Y\to X$ is a spectral subspace then $Y$ is noetherian. In other words, the class of noetherian spectral spaces is subspace closed.
\end{lem}
\begin{proof}
The space $Y$ is noetherian if and only if every open subset of $Y$ is quasi-compact. Let $U$ be an open subset of $Y$. Then there exists an open $U'$ in $X$ such that $U = i^{-1}(U')$. As $X$ is noetherian $U'$ is quasi-compact and since $i$ is spectral this implies $U$ is also quasi-compact. Since $U$ was arbitrary this shows $Y$ is noetherian.
\end{proof}

\begin{lem}\label{lem_noeth_dim}
Let $\mathcal{X}$ be the subspace closed class of noetherian spectral spaces. Then $\{\kdim_X \; \vert \; X\in \mathcal{X}\}$ is a class of spectral dimension functions compatible with $\mathcal{X}$.
\end{lem}
\begin{proof}
For each $X\in \mathcal{X}$ the last two lemmas show $\kdim_X$ is a spectral dimension function on $X$ and the compatibility condition is immediate from the definition of $\kdim_X$.
\end{proof}

We also easily verify that our criterion Theorem~\ref{thm_abstract} applies.

\begin{lem}\label{lem_noeth_base}
Let $\sfT$ be a rigidly-compactly generated tensor triangulated category with noetherian spectrum $\Spc\sfT^c = X$. Then for all $A\in \sfT$ we have
\begin{displaymath}
\GGamma_{X_{\leq 0}}A \in \Loc( \GGamma_xA\; \vert \; x\in X_{\leq 0}),
\end{displaymath}
where $X_{\leq 0}$ is defined with respect to $\kdim_X$.
\end{lem}
\begin{proof}
By definition of the Krull dimension $X_{\leq 0}$ consists precisely of the closed points of $X$. As $X$ is noetherian every closed subset is Thomason and so, in particular, for every $x\in X_{\leq 0}$ the subset $\{x\}$ is Thomason. Thus we can conclude by applying Lemma~\ref{lem_theone}. 
\end{proof}

Applying Theorem~\ref{thm_abstract} and Corollary~\ref{cor_abstract} we obtain the following result (cf.\ \cite{StevensonActions}*{Theorem~6.9}).

\begin{thm}\label{thm_noeth}
Suppose $\sfT$ is a rigidly-compactly generated tensor triangulated category with noetherian spectrum. Assume $\sfT$ satisfies at least one of the following two conditions:
\begin{itemize}
\item[(1)] $\sfT$ has a monoidal model;
\item[(2)] $\kdim_{\Spc \sfT^c}\Spc \sfT^c < \omega + \omega.$
\end{itemize}
Then for every object $A$ of $\sfT$ we have
\begin{displaymath}
A \in \Loc( \GGamma_xA \; \vert \; x\in \supp A ).
\end{displaymath}
In particular, the local-to-global principle holds for the action of $\sfT$ on itself and for every $A\in \sfT$ we have $\supp A = \varnothing$ if and only if $A\iso 0$. Furthermore these properties hold for the action of $\sfT$ on any other compactly generated triangulated category.
\end{thm}
\begin{proof}
As stated before the theorem, given the results of this section one can immediately apply either Theorem~\ref{thm_abstract} or Corollary~\ref{cor_abstract}. As noted in Corollary~\ref{cor_ltg} the usual statement of the local-to-global principle follows immediately. One then proves the remaining statements exactly as in \cite{StevensonActions}*{Theorem~6.9}.
\end{proof}

\section{Categories with constructible spectrum}
We now provide a new application of Theorem~\ref{thm_abstract} generalising \cite{StevensonAbsFlat}*{Lemma~4.20} and extending the classification result obtained in said paper to derived categories of arbitrary semi-artinian absolutely flat rings. Throughout we make a great deal of use of the constructible topology and related concepts; some reminders on these concepts can be found in Section~\ref{sec_spectralprelims}. We start with some recollections on the Cantor-Bendixson rank of a space.

\begin{defn}\label{defn_cbrank}
Let $X$ be a spectral space. Denote by $X_{\leq 0}$ the set of isolated, i.e.\ open, points in $X$ and let $X_{>0} = X\setminus X_{\leq 0}$. Suppose $X_{\leq \alpha}$ has been defined and denote by $X_{>\alpha}$ the complement of $X_{\leq \alpha}$. We define
\begin{displaymath}
X_{\leq \alpha+1} = X_{\leq \alpha} \cup \{x\in X_{>\alpha}\; \vert \; x \; \text{is open in} \; X_{> \alpha}\}.
\end{displaymath}
For a limit ordinal $\lambda$ we set
\begin{displaymath}
X_{\leq \lambda} = \bigcup_{\kappa < \lambda} X_{\leq \kappa}.
\end{displaymath}
If there is an ordinal $\alpha$ such that $X_{\leq \alpha} = X$ then the least such ordinal $\beta$ is the \emph{Cantor-Bendixson rank} of $X$. In this case we write $\cbrk X = \beta$ and say $X$ has Cantor-Bendixson rank. If there is no such ordinal we say the Cantor-Bendixson rank of $X$ is undefined.

If $X$ has Cantor-Bendixson rank this gives rise to a dimension function on $X$ by setting, for $x\in X$,
\begin{displaymath}
\cbrk(x) = \min\{\alpha \in \Ord \; \vert \; x\notin X_{>\alpha}\}.
\end{displaymath}
We note that compatibility of $\cbrk$ with the class of spectral spaces with Cantor-Bendixson rank is immediate from the definition.
\end{defn}

\begin{lem}\label{lem_CBclosed}
If $X$ is a spectral space with Cantor-Bendixson rank and $i\colon Y\to X$ is a spectral subspace then $Y$ has Cantor-Bendixson rank. In other words, the class of spectral spaces with defined Cantor-Bendixson rank is subspace closed.
\end{lem}
\begin{proof}
We claim that for every ordinal $\alpha$ there exists an $\alpha' \geq \alpha$ such that $Y_{>\alpha} \subseteq X_{>\alpha'}$. This suffices since, by the assumption that $X$ has Cantor-Bendixson rank, there exists an ordinal $\beta$ with $X_{>\beta} = \varnothing$ yielding
\begin{displaymath}
\exists\; \beta'\geq\beta \; \text{such that} \; Y_{>\beta} \subseteq X_{>\beta'} = \varnothing.
\end{displaymath}
This implies $Y_{>\beta} = \varnothing$ and so $\cbrk Y \leq \beta$.

We now prove the claim by transfinite induction. For the base case, let $\alpha'$ be the least ordinal such that $X_{\leq\alpha'} \cap Y \neq \varnothing$ i.e.\ the least Cantor-Bendixson rank in $X$ of a point in $Y$. We note that $\alpha'$ must be a successor ordinal. Indeed, if $\alpha'$ were a limit ordinal then we would have
\begin{displaymath}
\varnothing \neq X_{\leq \alpha'} \cap Y = (\bigcup_{\kappa<\alpha'} X_{\leq \kappa})\cap Y = \bigcup_{\kappa<\alpha'}(X_{\leq \kappa} \cap Y),
\end{displaymath}
showing there exists a $\kappa < \alpha'$ with $X_{\leq \kappa}\cap Y\neq \varnothing$ contradicting the minimality of $\alpha'$. So $\alpha'-1$ exists and satisfies $X_{\leq\alpha'-1}\cap Y = \varnothing$ from which we deduce $Y_{> -1} = Y\subseteq X_{>\alpha'-1}$.

Next suppose $\alpha$ is an ordinal for which there exists an $\alpha'\geq\alpha$ with $Y_{>\alpha} \subseteq X_{>\alpha'}$. We can consider the inclusions
\begin{displaymath}
Y_{>\alpha+1} \subseteq Y_{>\alpha} \subseteq X_{>\alpha'}
\end{displaymath}
and apply the argument used for the base case to the composite inclusion. This furnishes us with an $\alpha''$ such that
\begin{displaymath}
Y_{>\alpha+1} \subseteq (X_{>\alpha'})_{>\alpha'' -1} = X_{>\alpha'+\alpha''-1}
\end{displaymath}
as required.

Finally, suppose $\lambda$ is a limit ordinal such that for every $\alpha<\lambda$ there is an $\alpha'>\alpha$ with $Y_{>\alpha} \subseteq X_{>\alpha'}$, and fix such a family of ordinals $\alpha'$. Then, since the $X_{>\alpha'}$ form a descending chain, we have
\begin{displaymath}
Y_{>\lambda} = \bigcap_{\alpha<\lambda} Y_{>\alpha} \subseteq \bigcap_{\alpha'}X_{>\alpha'} \subseteq X_{>\lambda'}
\end{displaymath}
where $\lambda'$ is the supremum of the ordinals $\alpha'$. Thus the claim holds, which proves the statement as indicated at the beginning of the proof.
\end{proof}

\begin{lem}
Let $X$ be a spectral space such that $X = X^\mathrm{con}$ and $X$ has Cantor-Bendixson rank. Then $\cbrk$ is a spectral dimension function on $X$.
\end{lem}
\begin{proof}
As $X = X^\mathrm{con}$ a subset of $X$ is Thomason if and only if it is open (see e.g.\ \cite{StevensonAbsFlat}*{Lemma~3.5}). We proceed by transfinite induction, the base case being trivial as $X_{\leq 0}$ is a union of open points. Suppose $\alpha \in \Ord$ and $X_{\leq \alpha}$ is open. By definition
\begin{displaymath}
X_{\leq \alpha+1} = X_{\leq \alpha} \cup \{x\in X_{>\alpha}\; \vert \; x \; \text{is open in} \; X_{> \alpha}\}.
\end{displaymath}
If $x\in X_{\leq \alpha+1}\setminus X_{\leq \alpha}$, that is to say is open in $X_{>\alpha}$, then there is an open neighbourhood $U$ of $x$ such that $U\cap X_{>\alpha} = \{x\}$ i.e., $U\setminus \{x\} \subseteq X_{\leq \alpha}$. Thus $U$ is an open neighbourhood of $x$ contained in $X_{\leq \alpha+1}$ showing $X_{\leq \alpha+1}$ is open. If $\lambda$ is a limit ordinal then
\begin{displaymath}
X_{\leq \lambda} = \bigcup_{\kappa < \lambda} X_{\leq \kappa}.
\end{displaymath}
By the induction hypothesis each $X_{\leq \kappa}$ is open and thus $X_{\leq \lambda}$ is also open.
\end{proof}

Now that we have shown the Cantor-Bendixson rank gives a family of spectral dimension functions on the subspace closed class of constructible spectral spaces with Cantor-Bendixson rank we are in a position to apply Theorem~\ref{thm_abstract} and prove the local-to-global principle. We do so after the following lemma, which is somewhat of an aside, demonstrating that isolated points in the spectrum of a rigidly-compactly generated tensor triangulated category produce particularly nice idempotent objects.


\begin{lem}\label{lem_absflat1}
Let $\sfT$ be a rigidly-compactly generated tensor triangulated category whose spectrum, $X=\Spc\sfT^c$, carries the constructible topology and has Cantor-Bendixson rank. Then for $x\in X_{\leq 0}$, where $X_{\leq 0}$ is defined in terms of the Cantor-Bendixson dimension function, 
the localisation triangle
\begin{displaymath}
\GGamma_{\{x\}}\mathbf{1} \to \mathbf{1} \to L_{\{x\}}\mathbf{1} \to \Sigma \GGamma_{\{x\}}\mathbf{1}
\end{displaymath}
is split. In particular the object $\GGamma_x\mathbf{1} \cong \GGamma_{\{x\}}\mathbf{1}$ is compact.
\end{lem}
\begin{proof}
Let $x\in X_{\leq 0}$ so $x$ is open i.e., Thomason in $X$. Since $x$ is also closed its complement $\mathcal{Z}(x)$ is also a Thomason subset of $X$. This observation gives rise to localisation triangles (corresponding to smashing localisations)
\begin{displaymath}
\GGamma_{\{x\}}\mathbf{1} \to \mathbf{1} \to L_{\{x\}}\mathbf{1} \to \Sigma \GGamma_{\{x\}}\mathbf{1} \quad \text{and} \quad \GGamma_{\mathcal{Z}(x)}\mathbf{1} \to \mathbf{1} \to L_{\mathcal{Z}(x)}\mathbf{1} \to \Sigma \GGamma_{\{x\}}\mathbf{1}.
\end{displaymath}
Tensoring the first triangle with $\GGamma_{\mathcal{Z}(x)}\mathbf{1}$ and the second with $L_{\{x\}}\mathbf{1}$ yields isomorphisms
\begin{displaymath}
L_{\{x\}}\mathbf{1} \iso \GGamma_{\mathcal{Z}(x)}\mathbf{1}\otimes L_{\{x\}}\mathbf{1} \iso \GGamma_{\mathcal{Z}(x)}\mathbf{1}.
\end{displaymath}
By the unicity of $\GGamma_x\mathbf{1}$ (see \cite{BaRickard}*{Corollary~7.5}) we have isomorphisms
\begin{displaymath}
\GGamma_{\{x\}}\mathbf{1} \iso \GGamma_x\mathbf{1} \iso L_{\mathcal{Z}(x)}\mathbf{1}.
\end{displaymath}
From these two collections of isomorphisms we see
\begin{displaymath}
\Hom(L_{\{x\}}\mathbf{1}, \Sigma \GGamma_{\{x\}}\mathbf{1}) \iso \Hom(\GGamma_{\mathcal{Z}(x)}\mathbf{1}, \Sigma L_{\mathcal{Z}(x)}\mathbf{1}) \iso 0.
\end{displaymath}
Thus the triangle
\begin{displaymath}
\GGamma_{\{x\}}\mathbf{1} \to \mathbf{1} \to L_{\{x\}}\mathbf{1} \to \Sigma \GGamma_{\{x\}}\mathbf{1}
\end{displaymath}
splits and $\GGamma_{\{x\}}\mathbf{1} \iso \GGamma_x\mathbf{1}$ is a summand of $\mathbf{1}$ and hence compact.
\end{proof}

\begin{lem}\label{lem_absflat2}
Let $\sfT$ be a rigidly-compactly generated tensor triangulated category whose spectrum, $X=\Spc\sfT^c$, carries the constructible topology and has Cantor-Bendixson rank. Given any $A\in \sfT$ we have
\begin{displaymath}
\GGamma_{X_{\leq 0}}A \in \Loc( \GGamma_xA \; \vert \; x\in X_{\leq 0}),
\end{displaymath}
where $X_{\leq 0}$ is defined with respect to $\cbrk$.
\end{lem}
\begin{proof}
Let $x$ be a point of $X_{\leq 0}$. Then, by definition, $\{x\}$ is an open subset of $X$. Thus its complement is closed and hence quasi-compact. This shows $\{x\}$, which is also a closed subset, is Thomason. Having verified this condition we can appeal to Lemma~\ref{lem_theone} to finish the proof.
\end{proof}

\begin{thm}\label{thm_absflat}
Suppose $\sfT$ is a rigidly-compactly generated tensor triangulated category whose spectrum carries the constructible topology and has Cantor-Bendixson rank. In addition suppose $\sfT$ satisfies at least one of the following two conditions:
\begin{itemize}
\item[(1)] $\sfT$ has a monoidal model;
\item[(2)] $\cbrk\Spc \sfT^c < \omega + \omega.$
\end{itemize}
Then for every object $A$ of $\sfT$ we have
\begin{displaymath}
A \in \Loc( \GGamma_xA \; \vert \; x\in \supp A ).
\end{displaymath}
In particular, the local-to-global principle holds for the action of $\sfT$ on itself and for every $A\in \sfT$ we have $\supp A = \varnothing$ if and only if $A\iso 0$. Furthermore these properties hold for the action of $\sfT$ on any other compactly generated triangulated category.
\end{thm}
\begin{proof}
We first note that, as recalled in Lemma~\ref{lem_consvis}, every point of a spectral space equipped with the constructible topology is visible. So, by what we have thusfar proved in this section, either Theorem~\ref{thm_abstract} or Corollary~\ref{cor_abstract} applies. The remaining statements are a consequence of Corollary~\ref{cor_ltg}.
\end{proof}

\section{An application to absolutely flat rings}

In \cite{StevensonAbsFlat} it was shown that the local-to-global principle holds for derived categories of certain absolutely flat rings (the definition of absolutely flat is recalled below) and does not hold in certain other cases. The theorem of the last section, Theorem~\ref{thm_absflat}, allows us to prove the local-to-global principle in the remaining cases (cf.\ Remark~4.22 of the aforementioned paper).

First we recall the relevant definitions; further details and motivation can be found in \cite{StevensonAbsFlat}.

\begin{defn}
Let $R$ be a commutative ring with unit. We say $R$ is \emph{absolutely flat} (also known as \emph{von Neumann regular}) if for every $r\in R$ there exists some $x\in R$ satisfying
\begin{displaymath}
r = r^2 x.
\end{displaymath}
\end{defn}

\begin{defn}\label{defn_semiartinian}
Let $S$ be a ring (not necessarily absolutely flat). We say $S$ is \emph{semi-artinian} if every non-zero homomorphic image of $S$, in the category of $S$-modules, contains a simple submodule.

\end{defn}

Given a prime ideal $P\in \Spec R$ we denote by $k(P)$ the corresponding residue field, which is just $R_P$ in the case $R$ is absolutely flat.

\begin{thm}\label{thm_app}
Let $R$ be a semi-artinian absolutely flat ring. Then $\D(R)$ satisfies the local-to-global principle. In particular the residue fields of $R$ generate the derived category $\D(R)$. Moreover, there is an order preserving bijection
\begin{displaymath}
\left\{ \begin{array}{c}
\text{subsets of}\; \Spec R 
\end{array} \right\}
\xymatrix{ \ar[r]<1ex>^{\tau} \ar@{<-}[r]<-1ex>_{\sigma} &} \left\{
\begin{array}{c}
\text{localising subcategories of} \; \D(R) \\
\end{array} \right\}, 
\end{displaymath}
where for a localising subcategory $\mathcal{L}$ and a subset $W\cie \Spec R$ we set
\begin{displaymath}
\sigma(\mathcal{L}) = \{ P\in \Spec R \; \vert \; k(P)\otimes \mathcal{L}\neq 0\} \quad \text{and} \quad \tau(W) = \Loc (k(P) \; \vert \; P\in W).
\end{displaymath}
\end{thm}
\begin{proof}
Let $R$ be a semi-artinian absolutely flat ring as in the statement. Then $\Spec R$ is certainly spectral and carries the constructible topology (see \cite{Olivier_UAF}*{Proposition~5} for example for the second statement). By \cite{Garavaglia}*{Theorem~4} (or see \cite{TrlifajSuperdecomposable} which is more easily obtained) the space $\Spec R \cong \Spc \D^\mathrm{perf}(R)$ has Cantor-Bendixson rank. Finally, it is a standard fact that $\D(R)$ is the homotopy category of a stable monoidal model category. So we can apply Theorem~\ref{thm_absflat} to deduce that $\D(R)$ satisfies the local-to-global principle.

Given a prime ideal $P\in \Spec R$ we have $k(P) \cong R_P \cong \GGamma_P R$ (see \cite{StevensonAbsFlat}*{Lemma~4.2} and the discussion following it) from which it is immediate, by an application of the local-to-global principle to the stalk complex $R$, that the residue fields generate. The classification result is also an immediate consequence of the local-to-global principle and follows as in the proof of \cite{StevensonAbsFlat}*{Theorem~4.23}.
\end{proof}

  \bibliography{greg_bib}

\begin{bibdiv}
\begin{biblist}

\bib{BaSpec}{article}{
      author={Balmer, Paul},
       title={The spectrum of prime ideals in tensor triangulated categories},
        date={2005},
     journal={J. Reine Angew. Math.},
      volume={588},
       pages={149\ndash 168},
}

\bib{BaFilt}{article}{
      author={Balmer, Paul},
       title={Supports and filtrations in algebraic geometry and modular
  representation theory},
        date={2007},
     journal={Amer. J. Math.},
      volume={129},
      number={5},
       pages={1227\ndash 1250},
}

\bib{BaRickard}{article}{
      author={Balmer, Paul},
      author={Favi, Giordano},
       title={Generalized tensor idempotents and the telescope conjecture},
        date={2011},
     journal={Proc. Lond. Math. Soc. (3)},
      volume={102},
      number={6},
       pages={1161\ndash 1185},
}

\bib{BIKStrat2}{article}{
      author={Benson, Dave},
      author={Iyengar, Srikanth~B.},
      author={Krause, Henning},
       title={Stratifying triangulated categories},
        date={2011},
     journal={J. Topol.},
      volume={4},
      number={3},
       pages={641\ndash 666},
}

\bib{BKS}{article}{
      author={Buan, Aslak~Bakke},
      author={Krause, Henning},
      author={Solberg, {\O}yvind},
       title={Support varieties: an ideal approach},
        date={2007},
     journal={Homology, Homotopy Appl.},
      volume={9},
      number={1},
       pages={45\ndash 74},
}

\bib{Garavaglia}{unpublished}{
      author={Garavaglia, S.},
       title={Dimension and rank in the model theory of modules},
        date={1979},
        note={Preprint, University of Michigan},
}

\bib{GrothPSD}{article}{
      author={Groth, M.},
       title={Derivators, pointed derivators, and stable derivators},
        date={2013},
     journal={Algebr. Geom. Topol.},
      volume={13},
       pages={313\ndash 374},
}

\bib{GKrause_Kdim}{article}{
      author={Krause, G{\"u}nter},
       title={On the {K}rull-dimension of left noetherian left {M}atlis-rings},
        date={1970},
     journal={Math. Z.},
      volume={118},
       pages={207\ndash 214},
}

\bib{NeeChro}{article}{
      author={Neeman, A.},
       title={The chromatic tower for {$D(R)$}},
        date={1992},
     journal={Topology},
      volume={31},
      number={3},
       pages={519\ndash 532},
        note={With an appendix by Marcel B{\"o}kstedt},
}

\bib{Olivier_UAF}{article}{
      author={Olivier, Jean-Pierre},
       title={Anneaux absolument plats universels et \'{e}pimorphismes \`{a}
  buts r\'{e}duits},
        date={1967},
     journal={S\'{e}minaire Samuel. Alg\'{e}bre commutative},
      volume={2},
       pages={1\ndash 12},
}

\bib{stacks-project}{misc}{
      author={{Stacks Project Authors}, The},
       title={\itshape stacks project},
        date={2015},
        note={\url{http://stacks.math.columbia.edu}},
}

\bib{StevensonFilt}{unpublished}{
      author={Stevenson, Greg},
       title={Filtrations via tensor actions},
        date={2012},
        note={arXiv:1206.2721},
}

\bib{StevensonActions}{article}{
      author={Stevenson, Greg},
       title={Support theory via actions of tensor triangulated categories},
        date={2013},
     journal={J. Reine Angew. Math.},
      volume={681},
       pages={219\ndash 254},
}

\bib{StevensonAbsFlat}{article}{
      author={Stevenson, Greg},
       title={Derived categories of absolutely flat rings},
        date={2014},
     journal={Homology Homotopy Appl.},
      volume={16},
      number={2},
       pages={45\ndash 64},
}

\bib{TrlifajSuperdecomposable}{incollection}{
      author={Trlifaj, Jan},
       title={Two problems of {Z}iegler and uniform modules over regular
  rings},
        date={1996},
   booktitle={Abelian groups and modules ({C}olorado {S}prings, {CO}, 1995)},
      series={Lecture Notes in Pure and Appl. Math.},
      volume={182},
   publisher={Dekker},
     address={New York},
       pages={373\ndash 383},
}

\end{biblist}
\end{bibdiv}

\end{document}